\documentclass[12pt]{amsart}

\makeatletter
\@namedef{subjclassname@2020}{%
  \textup{2020} Mathematics Subject Classification}
\makeatother

\usepackage{amsmath,amssymb,amsthm,mathrsfs}
\usepackage{hyperref}
\usepackage{mathrsfs}

\numberwithin{equation}{section}

\newtheorem{theorem}{Theorem}[section]
\newtheorem{proposition}[theorem]{Proposition}

\newtheorem{remark}[theorem]{Remark}
\newtheorem{example}[theorem]{Example}
\theoremstyle{definition}
\newtheorem{definition}[theorem]{Definition}

\renewcommand{\ge}{\geqslant}
\renewcommand{\le}{\leqslant}

\begin{document}

\title
[Polylogarithmic Analogues of Euler's Constant]
{Polylogarithmic Analogues of Euler's Constant}


\author[T. Noda]{Takumi Noda}
\address{College of Engineering\\ 
Nihon University \\
963--8642 Fukushima, Japan}
\email{noda.takumi@nihon-u.ac.jp}

\date{}

\begin{abstract}
We introduce a family of constants
\[
C_m
:=
\lim_{n\to\infty}
\left(
\sum_{k=1}^n
\operatorname{Li}_m\!\left(\frac1k\right)
-
\log n
\right),
\]
which may be regarded as polylogarithmic analogues
of Euler's constant.
We study their basic properties and derive
representations in terms of iterated logarithmic integral structures
associated with the gamma function.
We further introduce associated
polylogarithmic zeta potentials and
polylogarithmic gamma functions,
establish differential relations
and integral representations,
and describe logarithmic branch asymptotics
near the singular points.
As an application,
we relate the constants \(C_m\)
to special values of certain Dirichlet series
involving the Riemann zeta function.
\end{abstract}


\subjclass[2020]{Primary 33B30; Secondary 11M41}

\keywords{
Euler's constant,
polylogarithm,
Riemann zeta-function,
Barnes \(G\)-function,
logarithmic gamma integrals,
multiple gamma functions
}

\maketitle

\section{Introduction}
Among the classical constants
appearing in analytic number theory,
Euler's constant
\[
\gamma
=
\lim_{n\to\infty}
\left(
\sum_{k=1}^n \frac1k
-
\log n
\right)
\]
occupies a distinguished position
through its close connections with
the gamma function,
zeta functions,
and logarithmic phenomena.
For historical background
and modern developments concerning Euler's constant,
see Lagarias \cite{lagarias2013}.
Motivated by the classical definition of \(\gamma\),
for integers \(m\ge2\), 
we introduce the constants 
\begin{equation}
C_m
:=
\lim_{n\to\infty}
\left(
\sum_{k=1}^n
\operatorname{Li}_m\!\left(\frac1k\right)
-
\log n
\right),
\end{equation}
where
\[
\operatorname{Li}_m(z)
=
\sum_{n=1}^{\infty}
\frac{z^n}{n^m}
\]
denotes the polylogarithm.
The constants \(C_m\)
may naturally be viewed as
polylogarithmic analogues of Euler's constant.
Indeed,
the logarithmic divergence
of the partial sums
is analogous to the classical harmonic case,
while the polylogarithmic structure introduces
new connections with zeta values,
iterated integrals,
and logarithmic gamma-type functions.

One of the fundamental identities satisfied by \(C_m\) is
\[
\sum_{k=2}^{\infty}
\frac{\zeta(k)-1}{k^m}
=1-\gamma-\zeta(m)+C_m,
\]
which connects the constants \(C_m\)
with Dirichlet series involving
the Riemann zeta function $\zeta(s)$.
Related Dirichlet series
were studied in \cite{noda2026dirichlet}.
This identity may be viewed as
a direct extension
of Euler's classical identity
\[
\sum_{k=2}^{\infty}
\frac{\zeta(k)-1}{k}
=
1-\gamma.
\]

Motivated by this identity,
we introduce two associated function families:
polylogarithmic zeta potentials
and polylogarithmic gamma functions.
Although the latter are not
Barnes multiple gamma functions
in the usual sense,
they are motivated by a related philosophy:
higher logarithmic structures arise through
recursive procedures associated with
the gamma function.
\[
\begin{array}{ccc}
\textbf{Barnes hierarchy}
&
&
\textbf{Present hierarchy}
\\[0.3em]
\text{difference equations}
& \leadsto &
\text{integral recursion}
\\[0.3em]
\text{factorial recursion}
& \leadsto &
\text{polylogarithmic recursion}
\\[0.3em]
\text{multiple gamma functions}
& \leadsto &
\text{iterated logarithmic gamma functions}
\end{array}
\]
The present framework is distinguished by
its recursive logarithmic-gamma realization,
its associated logarithmic branch structures,
and its connection with
polylogarithmic analogues of Euler's constant.

\medskip

The paper is organized as follows.
In Section~2,
we introduce the polylogarithmic zeta potentials
\[
\mathcal Z_s(z)
:=
\sum_{n=1}^{\infty}
\left\{
\operatorname{Li}_s\!\left(\frac{z}{n}\right)
-
\frac{z}{n}
\right\},
\]
which serve as generating functions
for the constants introduced above.
We study their logarithmic gamma structure
and analytic continuation,
together with the logarithmic branch singularities
arising at
\(
z=1,2,3,\dots
\).
In Section~3,
we introduce the polylogarithmic gamma functions
\(
G_s(z)
\),
which satisfy
\[
G_1(z)
=
e^{-\gamma z}\Gamma(1-z).
\]
We derive iterated integral representations,
differential relations,
and special-value formulas relating \(G_s(z)\)
to the constants \(C_m\).
We also discuss their relation
to classical logarithmic gamma integrals,
including connections with
the Kinkelin function
and Barnes-type multiple gamma functions.
In Section~4,
we investigate arithmetic properties
of the constants \(C_m\),
including asymptotic and summation formulas,
and their relation
to Dirichlet series involving
the Riemann zeta function.

\section{Polylogarithmic Zeta Potentials}
\subsection{Definition and power series expansion}
We begin by introducing
a family of generating functions
built from polylogarithms and zeta values.
\begin{definition}[Polylogarithmic zeta potentials]
For \(s\in\mathbb C\) and \(|z|<1\),
define
\begin{equation}
\mathcal Z_s(z)
:=
\sum_{n=1}^{\infty}
\left\{
\operatorname{Li}_s\!\left(\frac{z}{n}\right)
-
\frac{z}{n}
\right\}.
\end{equation}
\end{definition}
The terminology ``potential''
is used heuristically,
emphasizing the role of
\(\mathcal Z_s(z)\)
as a generating analytic object
encoding zeta values through polylogarithmic layers.
Indeed, the generating structure
becomes apparent
upon expanding the polylogarithm term-by-term:
\[
\operatorname{Li}_s\!\left(\frac{z}{n}\right)
-
\frac{z}{n}
=
\sum_{r=2}^{\infty}
\frac{z^r}{r^s n^r}.
\]
Summing over \(n\),
and interchanging the order of summation
in the region of absolute convergence,
we obtain
\begin{equation}
\mathcal Z_s(z)
=
\sum_{r=2}^{\infty}
\frac{\zeta(r)}{r^s}z^r.
\end{equation}
Hence,
for each fixed \(s\in\mathbb C\),
the function \(\mathcal Z_s(z)\)
is analytic for \(|z|<1\).
Moreover,
since \(\zeta(r)\to1\) exponentially fast
as \(r\to\infty\),
while \(r^{-s}\) contributes only polynomial growth,
the associated power series
has radius of convergence equal to \(1\)
by the Cauchy--Hadamard formula.

Evaluating the polylogarithmic zeta potentials at \(z=1\),
we obtain the associated arithmetic constants.
\begin{definition}[Polylogarithmic Euler constants]
For \(\Re(s)>1\), define
\begin{equation}
C_s
:=
\gamma+\mathcal Z_s(1)
=
\gamma+
\sum_{r=2}^{\infty}
\frac{\zeta(r)}{r^s}.
\end{equation}
For integers \(m\ge2\),
this recovers the constants defined in (1.1).
Equivalently,
\[
C_s
=
\sum_{n=1}^{\infty}\frac{a_n}{n^s},
\qquad
a_n=
\begin{cases}
\gamma & (n=1),\\
\zeta(n) & (n\ge2).
\end{cases}
\]
\end{definition}
Thus the constants \(C_s\)
form a zeta-weighted Dirichlet series.
The following proposition
describes the meromorphic continuation
of \(\mathcal Z_s(1)\).
\begin{proposition}
The function \(\mathcal Z_s(1)\),
and therefore \(C_s\),
extends meromorphically
to the whole complex $s$-plane
with a simple pole at \(s=1\).
\end{proposition}
\begin{proof}
Separating the constant term in \(\zeta(r)\),
we obtain
\[
\mathcal Z_s(1)
=
\zeta(s)-1
+
\sum_{r=2}^{\infty}
\frac{\zeta(r)-1}{r^s}.
\]
Since
\[
\zeta(r)-1
=
O(2^{-r}),
\]
the latter series defines
an entire function of \(s\).
The conclusion follows.
\end{proof}
By analytic continuation,
\(C_s\) will also denote
the resulting meromorphic function.
%
\subsection{Logarithmic gamma structure and analytic continuation}

The special value \(s=1\)
plays a distinguished role.
Indeed,
the classical expansion
\cite[1.17(2)]{erdelyi1953}
\begin{equation}
\log\Gamma(1-z)
=
\gamma z
+
\sum_{r=2}^{\infty}
\frac{\zeta(r)}{r}z^r,
\qquad |z|<1,
\end{equation}
yields
\begin{equation}
\mathcal Z_1(z)
=
\log\Gamma(1-z)-\gamma z,
\qquad |z|<1.
\end{equation}

Define the integral operator
\begin{equation}
(\mathcal I f)(z)
:=
\int_0^z \frac{f(w)}{w}\,dw.
\end{equation}
For integers \(m\ge2\),
the functions \(\mathcal Z_m(z)\)
satisfy the recursive integral relation
\[
\mathcal Z_m(z)
=
(\mathcal I^{m-1}\mathcal Z_1)(z).
\]
\begin{proposition}
The function \(\mathcal Z_1(z)\)
admits an analytic continuation to
\[
\Omega:=\mathbb C\setminus[1,\infty),
\]
where \(\log\Gamma(1-z)\)
is defined using the principal branch
of the logarithm.
For every integer \(m\ge2\),
the function \(\mathcal Z_m(z)\)
extends holomorphically to \(\Omega\)
and admits the finite value
\[
\mathcal Z_m(1)=C_m-\gamma.
\]
In particular,
for every integer \(m\ge2\),
the iterated integral structure regularizes
the logarithmic singularity at \(z=1\),
so that \(\mathcal Z_m(z)\)
extends holomorphically across \(z=1\).
No additional branch points are created,
and the branch structure is inherited from
\(\log\Gamma(1-z)\).
\end{proposition}
%
\begin{proof}
Since \(\Gamma(1-z)\)
has poles only at
\(
z=1,2,3,\dots
\)
and has no zeros, the function
\(\Gamma(1-z)\) is holomorphic and nonvanishing on
\[
\Omega:=\mathbb C\setminus[1,\infty).
\]
The slit \([1,\infty)\) is introduced
to avoid the poles and fix a single-valued branch
of the logarithm.
Since \(\Omega\) is simply connected, there exists a holomorphic branch of
\[
\log\Gamma(1-z)
\]
on \(\Omega\), 
obtained by analytic continuation
from the principal branch near \(z=0\).
By the expansion (2.4) of \(\log\Gamma(1-z)\) at \(z=0\),
we obtain
\[
\mathcal Z_1(z)=O(z^2)
\qquad (z\to0).
\]
Hence \(\mathcal Z_1(z)/z\) extends holomorphically to \(z=0\).
More generally, if \(f\) is holomorphic on \(\Omega\) and satisfies
\[
f(z)=O(z^2)
\qquad (z\to0),
\]
then \(f(z)/z\) is holomorphic on \(\Omega\). 
Since \(\Omega\)
is simply connected, the integral
\[
(\mathcal I f)(z)
=
\int_0^z \frac{f(w)}{w}\,dw
\]
is path-independent and defines a holomorphic function on \(\Omega\).
Moreover, since \(f(w)/w = O(w)\),
\[
(\mathcal I f)(z)=O(z^2)
\qquad (z\to0).
\]
Thus the assertion follows by induction.
For \(m\ge2\),
the series (2.2)
converges at \(z=1\),
and therefore
\[
\mathcal Z_m(1)
=
\sum_{r=2}^{\infty}
\frac{\zeta(r)}{r^m}
=
C_m-\gamma,
\]
which agrees with the definition (2.3) of \(C_m\).
\end{proof}

The relation between
\(\mathcal Z_1(z)\)
and the logarithmic gamma function
provides the motivation for
the polylogarithmic gamma hierarchy
introduced in Section~3.
\subsection{Singularity structure of \(\mathcal Z_s(z)\)}
The preceding discussion
describes the local analytic behaviour
of \(\mathcal Z_s(z)\) near \(z=0\),
together with its analytic continuation
and logarithmic gamma structure.
We now discuss,
from a heuristic analytic viewpoint,
the expected singularity structure
of \(\mathcal Z_s(z)\).

From the defining series (2.1),
we formally expect that
the singularities of \(\mathcal Z_s(z)\)
are inherited from
the branch structure
of the individual polylogarithm terms.
Upon analytic continuation
in the variable \(z\),
this suggests the appearance
of singularities
at the positive integers
\[
z=1,2,3,\ldots.
\]
These singularities are not poles in the usual sense,
but branch singularities inherited from
the polylogarithmic structure.
Indeed, the term corresponding to \(n=k\) contains
\[
\operatorname{Li}_s\!\left(\frac zk\right),
\]
which possesses the classical branch point
of the polylogarithm at \(1\).
Consequently,
the singularity at \(z=k\)
is expected to exhibit
a polylogarithmic branch behaviour.

More precisely,
for non-integral \(s\),
the local singular behaviour near \(z=k\)
is governed by
\[
\Gamma(1-s)
\left(
-\log\frac zk
\right)^{s-1},
\]
in accordance with
the expansion of the polylogarithm
near \(1\)
(cf.~\cite[(9.3)]{wood1992};
see also \cite{lewin1981})).
The case \(s=1\) is exceptional,
since
\[
\operatorname{Li}_1(z)
=
-\log(1-z),
\]
so that the singularity reduces
to a pure logarithmic branch point.
For integral values \(s\ge2\),
the fractional logarithmic branch structure
reduces to ordinary logarithmic branch singularities.
Thus, from a global viewpoint,
the sequence
\[
z=1,2,3,\ldots
\]
may be regarded heuristically
as locations of
polylogarithmic branch points.
This viewpoint motivates
the use of the slit domain
\[
\Omega=\mathbb C\setminus[1,\infty),
\]
which avoids the branch singularities
arising from the polylogarithmic structure.
%
%
\subsection{Logarithmic asymptotics near \(z=1\)}
The logarithmic singularity of
\(\mathcal Z_1(z)\) at \(z=1\)
is inherited by the functions
\(\mathcal Z_m(z)\)
through the iterated integral construction,
with the singular behaviour becoming progressively weaker
as \(m\) increases.
More precisely,
iterated integration transforms the
logarithmic divergence of
\(\mathcal Z_1(z)\)
into weaker logarithmic branch terms
with progressively higher vanishing order.

As \(z\to1^{-}\),
the classical pole behaviour
\[
\Gamma(1-z)\sim \frac1{1-z}
\]
implies
\[
\log\Gamma(1-z)
=
-\log(1-z)+O(1).
\]
Combining this with (2.5), we obtain
\begin{equation}
\mathcal Z_1(z)
=
-\log(1-z)+O(1)
\qquad
(z\to1^-).
\end{equation}
Moreover,
since
\[
\mathcal Z_m
=
\mathcal I^{m-1}\mathcal Z_1,
\]
repeated integration progressively weakens
the logarithmic singularity at \(z=1\).
Formally,
successive integrations
of the leading logarithmic term
\[
-\log(1-z)
\]
produce factors of \(1-z\)
together with milder logarithmic terms,
thereby softening the singular behaviour
through the iterated integral hierarchy.

For example,
since
\[
\mathcal Z_2(z)
=
\int_0^z \frac{\mathcal Z_1(w)}{w}\,dw,
\]
we have
\[
\mathcal Z_2(1)-\mathcal Z_2(1-\varepsilon)
=
\int_{1-\varepsilon}^{1}
\frac{\mathcal Z_1(w)}{w}\,dw.
\]
Using (2.7), we obtain
\[
\mathcal Z_2(1)-\mathcal Z_2(1-\varepsilon)
=
\int_{1-\varepsilon}^{1}
\frac{-\log(1-w)+O(1)}{w}\,dw.
\]
Putting \(t=1-w\), this becomes
\[
\int_0^\varepsilon
\frac{-\log t+O(1)}{1-t}\,dt
=
\int_0^\varepsilon
\left(
-\log t+O(1)
\right)\,dt
=
\varepsilon\log\frac1\varepsilon
+
O(\varepsilon).
\]
Since
\[
\mathcal Z_2(1)=C_2-\gamma,
\]
we obtain
\[
\mathcal Z_2(1-\varepsilon)
=
C_2-\gamma
-
\varepsilon\log\frac1\varepsilon
+
O(\varepsilon),
\qquad
(\varepsilon\to0^+).
\]
Thus,
iterated integration weakens
the logarithmic divergence.
Although the divergent logarithmic singularity
occurs only for \(m=1\),
logarithmic branch behaviour
persists for all \(m\ge2\).
%
%
%
%
%
%
\section{Polylogarithmic Gamma Hierarchies}
\subsection{Definition and Recursive Integral Structure}
In view of (2.4) and (2.6),
it is natural to construct
a hierarchy generated
from the logarithmic gamma function
through repeated application
of the logarithmic integral operator.
Motivated by this structure,
we introduce
the following hierarchy
of polylogarithmic gamma functions.
\begin{definition}[Polylogarithmic gamma functions]
For \(s\in\mathbb C\),
define
\begin{equation}
G_s(z)
:=
\exp\left(
\sum_{r=2}^{\infty}
\frac{\zeta(r)}{r^s}z^r
\right),
\qquad |z|<1.
\end{equation}
\end{definition}
\begin{remark}
For each fixed \(z\) with \(|z|<1\),
the defining series converges absolutely
and defines an entire function of \(s\).
By definition (3.1),
\[
\log G_m(z)=\mathcal Z_m(z)
\]
together with Proposition~2.2
yields an analytic continuation
of \(G_m(z)\)
to the domain \(\Omega\).
\end{remark}
The terminology
``polylogarithmic gamma functions''
reflects the identity
\[
G_1(z)
=
e^{-\gamma z}\Gamma(1-z),
\]
together with the recursive structure
induced by repeated application
of the logarithmic integral operator
\[
(\mathcal I f)(z)
=
\int_0^z \frac{f(w)}{w}\,dw.
\]
This is analogous to the classical hierarchy
of polylogarithms,
whose adjacent orders are related
through the Euler differential operator
\[
\left(
z\frac{d}{dz}
\right)
\operatorname{Li}_{s+1}(z)
=
\operatorname{Li}_s(z).
\]
Since
\[
\log G_s(z)
=
\mathcal Z_s(z)
\]
by definition (3.1),
the family satisfies
the recursive relation
\[
\log G_{s+1}(z)
=
\mathcal I
(\log G_s)(z).
\]

Thus,
the functions \(G_s(z)\)
may be regarded as
a polylogarithmic hierarchy
generated from
the logarithmic gamma function.
For integer values \(m\ge2\),
the functions \(G_m(z)\)
are obtained from the base function
\(G_1(z)\)
through repeated application
of the operator \(\mathcal I\).
Indeed,
\[
\log {G}_m(z)
=
\mathcal I^{m-1}
(\log {G}_1)(z).
\]
For example,
\[
\log {G}_2(z)
=
\int_0^z
\frac{
\log\Gamma(1-w)-\gamma w
}{w}
\,dw,
\]
while
\[
\log {G}_3(z)
=
\int_0^z
\frac1u
\left(
\int_0^u
\frac{
\log\Gamma(1-w)-\gamma w
}{w}
\,dw
\right)\,du.
\]
In view of its logarithmic gamma integral structure,
the function \(G_2(z)\)
may be regarded as
a polylogarithmic analogue
of the Barnes \(G\)-function.
Further comparisons with
the classical Barnes hierarchy
will be given in Section~3.3.
%
%
\subsection{Integral Representations and Differential Relations}
The iterated logarithmic integral representation
admits a natural extension
to complex parameters \(s\)
with \(\Re(s)>1\).
\begin{theorem}[Integral representation of \(G_s(z)\)]
For \(\Re(s)>1\) and \(z\in\Omega\),
\[
\log G_s(z)
=
\frac1{\Gamma(s-1)}
\int_0^z
\frac{
\log\Gamma(1-w)-\gamma w
}{w}
\left(
\log\frac zw
\right)^{s-2}
\,dw,
\]
where the integral is taken along the line segment
joining \(0\) and \(z\).
\end{theorem}
\begin{remark}
By Definition~3.1, 
\(G_s(z)\) is entire in \(s\) for each fixed \(z\) with \(|z|<1\).
The condition \(\Re(s)>1\)
is required for the integral representation,
which provides an analytic continuation
of \(G_s(z)\) to \(\Omega\).
In particular,
the value at \(z=1\)
is well defined for \(\Re(s)>1\),
and is obtained by taking the radial limit
\(z\to1^{-}\).
\end{remark}
\begin{proof}
Temporarily, we assume \(|z|<1\).
By (2.4)
\[
\frac{
\log\Gamma(1-w)-\gamma w
}{w}
=
\sum_{r=2}^{\infty}
\frac{\zeta(r)}{r}w^{r-1}.
\]
Since the series converges absolutely for \(|w|<1\)
and the kernel
\[
\left(\log\frac zw\right)^{s-2}
\]
is integrable along the line segment for \(\Re(s)>1\),
term-by-term integration is justified.
Substituting the expansion above
into the right-hand side
of the asserted integral formula,
we get
\[
\frac1{\Gamma(s-1)}
\sum_{r=2}^{\infty}
\frac{\zeta(r)}{r}
\int_0^z
w^{r-1}
\left(
\log\frac zw
\right)^{s-2}
\,dw.
\]
Along the line segment joining \(0\) and \(z\), let
\[
w=ze^{-u}.
\]
Then
\[
dw=-ze^{-u}\,du,
\]
and hence
\[
\int_0^z
w^{r-1}
\left(
\log\frac zw
\right)^{s-2}
\,dw
=
z^r
\int_0^\infty
e^{-ru}u^{s-2}\,du.
\]
Since \(\Re(s)>1\),
the last integral equals
\[
\frac{\Gamma(s-1)}{r^{\,s-1}}.
\]
Therefore, we recover the following series: 
\[
\sum_{r=2}^{\infty}
\frac{\zeta(r)}{r}
\frac{z^r}{r^{\,s-1}}
=
\sum_{r=2}^{\infty}
\frac{\zeta(r)}{r^s}z^r
=
\log G_s(z).
\]
This proves the identity for \(|z|<1\).

It remains to show that the integral
extends to a well-defined holomorphic function
on \(\Omega\).
Near \(w=0\),
\[
\frac{\log\Gamma(1-w)-\gamma w}{w}
=
O(w).
\]
Hence the integrand is locally integrable at \(w=0\)
for every \(\Re(s)>1\).
If \(z\neq1\),
the integrand is analytic along the line segment
joining \(0\) and \(z\).
When \(z=1\),
using
\[
\log\Gamma(1-w)
=
-\log(1-w)+O(1)
\qquad (w\to1),
\]
and
\[
\log\frac1w
\sim 1-w,
\]
we obtain
\[
\frac{\log\Gamma(1-w)-\gamma w}{w}
\left(
\log\frac1w
\right)^{s-2}
=
O\!\left(
-\log(1-w)
(1-w)^{\Re(s)-2}
\right),
\]
which is integrable for \(\Re(s)>1\).

Therefore the integral is well defined for all
\(z\in\Omega\).
Moreover, the convergence is locally uniform on compact
subsets of \(\Omega\),
so the integral defines a holomorphic function on
\(\Omega\).
Since it agrees with
\(\log G_s(z)\)
for \(|z|<1\),
the identity theorem yields the asserted continuation.
\end{proof}
%
Thus,
the iterated logarithmic integral representation
admits a natural extension
to complex order parameters \(s\).
Moreover,
the recursive structure
yields a differential relation
between adjacent orders.
%
\begin{proposition}
For \(|z|<1\) and \(s\in\mathbb C\),
\[
z\frac{d}{dz}
\log G_{s+1}(z)
=
\log G_s(z).
\]
\end{proposition}
\begin{proof}
By definition,
\[
\log G_{s+1}(z)
=
\sum_{r=2}^\infty
\frac{\zeta(r)}{r^{s+1}}z^r.
\]
Differentiating term-by-term,
we obtain
\[
z\frac d{dz}\log G_{s+1}(z)
=
\sum_{r=2}^\infty
\frac{\zeta(r)}{r^s}z^r
=
\log G_s(z),
\]
which proves the assertion.
\end{proof}
%
By Proposition~2.3,
the polylogarithmic constants \(C_s\)
extend meromorphically
to the whole complex \(s\)-plane
with a simple pole at \(s=1\).
Thus, evaluating at \(z=1\),
the polylogarithmic gamma functions
recover the arithmetic constants
introduced earlier.
\begin{theorem}[Integral representation of the polylogarithmic constants]
The identity
\[
C_s=\log G_s(1)+\gamma
\]
yields a meromorphic continuation of
\(G_s(1)\)
to \(s\in\mathbb C\setminus\{1\}\).
For \(\Re(s)>1\),
\[
C_s
=
\frac1{\Gamma(s-1)}
\int_0^1
\frac{\log\Gamma(1-t)}{t}
(-\log t)^{s-2}
\,dt .
\]
\end{theorem}
\begin{proof}
By Definition~2.2 and the defining series
for \(G_s(z)\),
evaluating at \(z=1\) gives
\[
C_s=\log G_s(1)+\gamma.
\]
By Theorem~3.3 and Remark~3.4,
the radial limit \(z\to1^{-}\) exists and yields
\[
\log G_s(1)
=
\frac1{\Gamma(s-1)}
\int_0^1
\frac{
\log\Gamma(1-t)-\gamma t
}{t}
(-\log t)^{s-2}
\,dt .
\]
Hence
\[
\begin{aligned}
\log G_s(1)
&=
\frac1{\Gamma(s-1)}
\int_0^1
\frac{\log\Gamma(1-t)}{t}
(-\log t)^{s-2}
\,dt
\\
&\quad
-
\frac{\gamma}{\Gamma(s-1)}
\int_0^1
(-\log t)^{s-2}
\,dt .
\end{aligned}
\]
Using
\[
\int_0^1
(-\log t)^{s-2}
\,dt
=
\Gamma(s-1),
\]
we obtain
\[
\log G_s(1)
=
\frac1{\Gamma(s-1)}
\int_0^1
\frac{\log\Gamma(1-t)}{t}
(-\log t)^{s-2}
\,dt
-\gamma .
\]
Combining this with the identity
\(
C_s=\log G_s(1)+\gamma
\),
we obtain the asserted integral representation.
\end{proof}
%
%
\subsection{Comparison with the Kinkelin function}
The present construction may be compared
with logarithmic gamma integrals
appearing in the theory of
the Kinkelin function
and Barnes-type multiple gamma functions
(cf.~\cite{barnes1901,barnes1904,vardi1988,ruijsenaars2000,kurokawa2007}).
Recall that the Kinkelin function,
which is closely related to the Barnes \(G\)-function,
admits the logarithmic integral representation
\[
\log G(x)
:=
\int_0^x \log\Gamma(t)\,dt
+
\frac{x(x-1)}2
-
\frac{x}{2}\log(2\pi),
\]
whereas the present function satisfies
\[
\log G_2(x)
=
\int_0^x
\frac{
\log\Gamma(1-t)-\gamma t
}{t}
\,dt .
\]
Thus,
both constructions arise from
logarithmic gamma integrals,
although the present construction
contains a singular logarithmic kernel
reflecting the underlying polylogarithmic structure.
This analogy may be summarized schematically as
\[
\log G(x)
\qquad\leadsto\qquad
\log G_2(x),
\]
\[
\int_0^x \log\Gamma(t)\,dt
\qquad\leadsto\qquad
\int_0^x
\frac{\log\Gamma(1-t)}{t}\,dt.
\]
At the level of special values,
Raabe's formula
(cf.~\cite{raabe1844}),
\[
\int_0^1 \log\Gamma(x)\,dx
=
\frac12\log(2\pi)
\]
corresponds to
\[
\int_0^1
\frac{\log\Gamma(1-x)}{x}\,dx
=
C_2.
\]
Symbolically,
\[
\frac12\log(2\pi)
\qquad\leadsto\qquad
C_2.
\]
In this sense,
the constants \(C_s\)
and the functions \(G_s(z)\)
may be viewed as
polylogarithmic analogues
of classical logarithmic gamma constructions.
The analogy with the Kinkelin function
should be understood at the level of
logarithmic gamma integrals
rather than factorial-type difference equations.

\subsection{The value at \(z=-1\)}
Evaluating at \(z=-1\) gives rise to
an alternating analogue
of the constants \(C_m\).
Indeed, for \(\Re(s)>1\),
\[
\log G_s(-1)
=
\sum_{r=2}^{\infty}
\frac{(-1)^r\zeta(r)}{r^s}.
\]
Rewriting the series, we obtain
\[
\log G_s(-1)
=
1-\eta(s)
+
\sum_{r=2}^{\infty}
\frac{(-1)^r(\zeta(r)-1)}{r^s},
\]
where \(\eta(s)\) denotes the Dirichlet eta function.
Since \(\eta(s)\) is entire and
\(\zeta(r)-1=O(2^{-r})\),
the right-hand side defines an entire function of \(s\).
Thus \(\log G_s(-1)\)
extends holomorphically
to the whole complex \(s\)-plane.

For integers \(m\ge2\),
define 
alternating analogues of the constants \(C_m\)
\[
\widehat C_m
:=
-
\lim_{n\to\infty}
\left(
\sum_{k=1}^n
\operatorname{Li}_m\!\left(-\frac1k\right)
+
\log n
\right).
\]
Expanding the polylogarithm,
we obtain
\[
\operatorname{Li}_m\!\left(-\frac1k\right)
=
-\frac1k
+
\sum_{r=2}^{\infty}
\frac{(-1)^r}{r^m k^r},
\]
and hence
\[
\widehat C_m
=
\gamma
-
\sum_{r=2}^{\infty}
\frac{(-1)^r\zeta(r)}{r^m}.
\]
Consequently,
\[
\log G_m(-1)
=
\gamma-\widehat C_m.
\]
Thus the values \(G_m(-1)\)
may be interpreted
as alternating counterparts
of the values \(G_m(1)\).
In particular,
since
\[
G_1(z)
=
e^{-\gamma z}\Gamma(1-z),
\]
we recover the classical identity 
(cf.~\cite[1.17(3)]{erdelyi1953}):
\[
\log G_1(-1)
=
\sum_{r=2}^{\infty}
\frac{(-1)^r \zeta(r)}{r}
=
\gamma.
\]
Equivalently,
\[
G_1(-1)=e^\gamma .
\]
%
%
\section{Arithmetic Properties of the Constants \(C_m\)}
\subsection{Asymptotic and summation properties}
Although explicit evaluations of the individual constants
\(C_m\)
appear difficult in general,
the family
\(\{C_m\}_{m\ge2}\)
exhibits several remarkable collective structures.
These include asymptotic properties,
summation identities,
and connections with shifted Stieltjes-type constants.

We begin by investigating
the asymptotic behaviour of \(C_m\)
as \(m\to\infty\).
The identity in Theorem 3.6,
\[
C_m-\gamma=\log G_m(1),
\]
shows that the asymptotic behaviour of
\(C_m-\gamma\)
is equivalent to that of
\(\log G_m(1)\).
\begin{proposition}[Asymptotic and summation formulas for \(C_m\)]
The polylogarithmic constants \(C_m\)
satisfy the asymptotic relation
\[
C_m-\gamma
\;\sim\;
\frac{\zeta(2)}{2^{m}}
\qquad (m\to\infty).
\]
Moreover,
\[
\begin{aligned}
&\sum_{m=2}^{\infty}(C_m-\gamma)
\;=\;
\gamma+\gamma_1(0)-\gamma_1(1),
\\
&\sum_{m=2}^{\infty}(-1)^m(C_m-\gamma)
\;=\;
\log\sqrt{2\pi}-\frac{\gamma}{2}.
\end{aligned}
\]
Here
\[
\gamma_1(a)
:=
\lim_{n\to\infty}
\left(
\sum_{k=1}^{n}\frac{\log k}{k+a}
-\frac12\log^2 n
\right)
\]
denotes a shifted analogue
of the first Stieltjes constant.
In particular,
\(\gamma_1(0)\)
coincides with the classical first Stieltjes constant \(\gamma_1\).
\end{proposition}
\begin{remark}
The summation identities above
may also be obtained
as special cases of Theorem~2.5
in \cite{noda2026dirichlet}.
For completeness,
we provide a direct proof
in the present setting.
\end{remark}
\begin{proof}
By the series representation (2.3) of \(C_m\), we have
\[
C_m-\gamma
=
\sum_{r=2}^{\infty}
\frac{\zeta(r)}{r^m}.
\]
Hence
\[
C_m-\gamma
=
\frac{\zeta(2)}{2^m}
+
\sum_{r=3}^{\infty}
\frac{\zeta(r)}{r^m}
=
\zeta(2)\cdot 2^{-m}
+
O(3^{-m}),
\]
which establishes the asymptotic formula.

We now prove the summation identities.
By absolute convergence, we may interchange the order
of summation and obtain
\[
\sum_{m=2}^{\infty}(C_m-\gamma)
=
\sum_{r=2}^{\infty}
\zeta(r)
\sum_{m=2}^{\infty}\frac1{r^m}
=
\sum_{r=2}^{\infty}
\frac{\zeta(r)}{r(r-1)}.
\]
Using the Dirichlet series expansion \(\zeta(r)\), we obtain
\[
\sum_{r=2}^{\infty}
\frac{\zeta(r)}{r(r-1)}
=
\sum_{k=1}^{\infty}
\sum_{r=2}^{\infty}
\frac{1}{r(r-1)k^r}.
\]
For \(0<x\le1\), we have
\[
\sum_{r=2}^{\infty}
\frac{x^r}{r(r-1)}
=
x+(1-x)\log(1-x),
\]
where the value at \(x=1\) is understood as a limit.
Therefore
\[
\sum_{m=2}^{\infty}(C_m-\gamma)
=
1+
\sum_{k=2}^{\infty}
\left\{
\frac1k
+
\frac{k-1}{k}\log\left(1-\frac1k\right)
\right\}.
\]
Taking partial sums up to \(N\), the right-hand side is
\[
H_N
+
\sum_{k=1}^{N-1}
\frac{\log k}{k(k+1)}
-
\frac{N-1}{N}\log N.
\]
Letting \(N\to\infty\), we obtain
\[
\sum_{m=2}^{\infty}(C_m-\gamma)
=
\gamma
+
\sum_{k=1}^{\infty}
\frac{\log k}{k(k+1)}.
\]
On the other hand, from the definition of \(\gamma_1(a)\),
\[
\gamma_1(0)-\gamma_1(1)
=
\sum_{k=1}^{\infty}
\left(
\frac{\log k}{k}
-
\frac{\log k}{k+1}
\right)
=
\sum_{k=1}^{\infty}
\frac{\log k}{k(k+1)}.
\]
Thus
\[
\sum_{m=2}^{\infty}(C_m-\gamma)
=
\gamma+\gamma_1(0)-\gamma_1(1).
\]

It remains to prove the alternating identity. Again by absolute
convergence,
\[
\sum_{m=2}^{\infty}(-1)^m(C_m-\gamma)
=
\sum_{r=2}^{\infty}
\zeta(r)
\sum_{m=2}^{\infty}
\left(-\frac1r\right)^m
=
\sum_{r=2}^{\infty}
\frac{\zeta(r)}{r(r+1)}.
\]
Using the Dirichlet series expansion \(\zeta(r)\), we obtain
\[
\sum_{r=2}^{\infty}
\frac{\zeta(r)}{r(r+1)}
=
\sum_{k=1}^{\infty}
\sum_{r=2}^{\infty}
\frac{1}{r(r+1)k^r}.
\]
For \(0<x\le1\), one has
\[
\sum_{r=2}^{\infty}
\frac{x^r}{r(r+1)}
=
\frac{x}{2}
+
\frac{1-x}{x}\log(1-x)
+
1-x,
\]
again with the value at \(x=1\) understood as a limit.
It follows that
\[
\sum_{m=2}^{\infty}(-1)^m(C_m-\gamma)
=
\frac12
+
\sum_{k=2}^{\infty}
\left\{
\frac1{2k}
+
(k-1)\log\left(1-\frac1k\right)
+
1-\frac1k
\right\}.
\]
Taking partial sums and simplifying, we find
\[
\sum_{r=2}^{\infty}
\frac{\zeta(r)}{r(r+1)}
=
\lim_{N\to\infty}
\left\{
N
-
\frac12 H_N
+
\log N!
-
N\log N
\right\}.
\]
By Stirling's formula,
\[
\log N!
=
N\log N-N+\frac12\log N+\log\sqrt{2\pi}+o(1),
\]
and since \(H_N=\log N+\gamma+o(1)\), this gives
\[
\sum_{m=2}^{\infty}(-1)^m(C_m-\gamma)
=
\log\sqrt{2\pi}-\frac{\gamma}{2}.
\]
This completes the proof.
\end{proof}
%
%
\subsection{Relation with Dirichlet series}
We recall the Dirichlet series
studied in \cite{noda2026dirichlet}:
\[
D(s_1,s_2;\lambda)
=
\sum_{k=1}^{\infty}
\frac{\zeta(s_1+k)-1}{(k+\lambda)^{s_2}}.
\]
This series may be regarded
as a generating function
for shifted zeta values.
The classical identity
\[
D(1,1;1)=1-\gamma
\]
goes back to Euler~\cite{euler1785}.
More recently,
special values
\(D(1,1;\mu)\)
for \(\mu=2,3\)
were evaluated by
Srivastava~\cite[(5.3)]{srivastava1988}
and
Choi--Srivastava~\cite[(4.40)]{choi-srivastava1999},
in terms of
the Stirling constant
and the Glaisher--Kinkelin constant:
\[
\begin{aligned}
D(1,1;2)
&=
\frac32
-
\frac{\gamma}{2}
-
\log A_0,
\\
D(1,1;3)
&=
\frac{11}{6}
-
\frac{\gamma}{3}
-
\log A_0
-
2\log A_1.
\end{aligned}
\]
Here
\[
A_0=\sqrt{2\pi}
\]
denotes the Stirling constant,
while \(A_1\)
is the Glaisher--Kinkelin constant.
More generally,
the constants \(A_j\) (\(j\ge2\))
are known as the Bendersky constants
(cf.~\cite{bendersky1933}).
In \cite{noda2026dirichlet},
the following general formula was obtained:
\[
D(1,1;\mu)
=
H_\mu
-
\frac{\gamma}{\mu}
-
\sum_{j=0}^{\mu-2}
\binom{\mu-1}{j}\log A_j.
\]
We now present several examples
for higher values \(m\ge2\).
\begin{example}
For \(m\ge2\),
we have
\[
\begin{aligned}
D(1,m;1)
&=
(C_m-\gamma)-(\zeta(m)-1),
\\
D(2,m;2)
&=
(C_m-\gamma)-(\zeta(m)-1)
-\frac{\zeta(2)-1}{2^m}.
\end{aligned}
\]
\end{example}
\begin{example}
Let
\[
I_0
:=
\int_0^1
\log\Gamma(x)\log(1-x)\,dx,
\qquad
I_1
:=
\int_0^1
x\log\Gamma(x)\log(1-x)\,dx.
\]
Then
\[
\begin{aligned}
D(1,2;2)
={}&
I_0
+
\log A_0
-
\frac{\gamma}{4}
+
\frac54
-
\zeta(2),
\\
D(1,2;3)
={}&
2I_0
-
2I_1
+
\frac12\log A_0
+
\log A_1
-
\frac{\gamma}{9}
+
\frac{49}{36}
-
\zeta(2).
\end{aligned}
\]
\end{example}
The latter identities may be derived
from classical generating function techniques
associated with the digamma function.
These examples are related to
logarithmic gamma moments.
For \(m\ge0\), let
\[
M_m
:=
\int_0^1 x^m\log\Gamma(x)\,dx.
\]
The moments \(M_m\) are naturally connected
with the constants \(C_m\).
\begin{example}
For the logarithmic gamma moments \(M_m\),
the constant \(C_2\) admits the representation
\[
C_2
=
\sum_{m=0}^{\infty} M_m.
\]
\end{example}
%
%
%
\section*{Concluding Remarks}
In this paper,
we introduced
polylogarithmic analogues
of Euler's constant
together with a family of
polylogarithmic gamma functions.
The resulting family
admits recursive integral structures
closely related to
logarithmic gamma integrals
appearing in the theory of
the Barnes \(G\)-function
and related logarithmic gamma constructions.

The present framework
is also related to
Dirichlet series involving shifted zeta values,
which arise naturally in the study of
generalized Euler-type constants.
At the same time,
the associated integral representations
give rise to new logarithmic gamma moments,
suggesting further analogies
with Barnes-type logarithmic gamma constructions.

Several problems remain open.
Among them are the detailed study
of the branch singularity structure
of the functions $\mathcal Z_s(z)$,
and the investigation of further
polylogarithmic operator hierarchies
generated from logarithmic special functions.
These questions suggest that
the present constructions
may represent only the first layer
of a broader theory
linking polylogarithms,
zeta values,
and logarithmic gamma hierarchies.

\medskip

Finally, we note an operator-theoretic interpretation
of the polylogarithmic gamma family constructed in this paper.
Since
\[
(\mathcal I z^r)(z)
=
\frac{z^r}{r},
\]
the zeta potentials satisfy the formal relation
\[
\mathcal Z_s
=
\mathcal I^s \mathcal Z_0.
\]
Recalling (2.5),
\[
\mathcal Z_1(z)
=
\log\Gamma(1-z)-\gamma z,
\]
we obtain
\[
\log G_s(z)
=
\mathcal Z_s(z)
=
\mathcal I^{\,s-1}\mathcal Z_1(z),
\]
where \(\mathcal I^{\,s-1}\)
is understood through
the complex-order integral representation
established in Section~3.
Together with the identity in Theorem~3.6
\[
C_s-\gamma
=
\mathcal Z_s(1),
\]
these relations yield the following formal scheme:
\[
\mathcal Z_0
\xrightarrow{\ \mathcal I\ }
\mathcal Z_1
=
\log\Gamma(1-z)-\gamma z
\xrightarrow{\ \mathcal I^{\,s-1}\ }
\log G_s
=
\mathcal Z_s
\xrightarrow{\ z=1\ }
C_s-\gamma.
\]
This viewpoint suggests that the hierarchy
is generated by repeated application
of 
the inverse of the Euler differential operator
\[
\mathcal I
=
\left(
z\frac{d}{dz}
\right)^{-1}.
\]
It is also natural to consider
the hierarchy at non-positive indices.
Indeed,
\[
\log G_0(z)
=
-z\psi(1-z)-\gamma z,
\]
showing that the level \(s=0\)
is naturally expressed
in terms of the digamma function
$\psi(z)$.
More generally,
one may expect the functions \(G_{-m}\)
to be related to higher polygamma functions.
This operator-theoretic viewpoint
deserves further investigation.


\end{document}